\documentclass[a4paper,12pt,reqno]{amsart}
\usepackage{eurosym}
\usepackage{amsfonts}
\usepackage{amsmath,amsthm,amssymb}

\nonstopmode \numberwithin{equation}{section}

\newtheorem{theorem}{Theorem}
\newtheorem{corollary}{Corollary}[section]

\newtheorem{lemma}{Lemma}[section]

\allowdisplaybreaks

\begin{document}
\title[Generalized fractional operator representations of Jacobi type..]{%
Generalized fractional operator representations of Jacobi type
orthogonal polynomials.}

\begin{abstract}
The aim of this paper is to apply generalized operators of fractional
integration and differentiation involving Appell's function $F_{3}(:)$ due
to Marichev-Saigo-Maeda (MSM), to the Jacobi type orthogonal polynomials. The
results are expressed in terms of generalized hypergeometric function. 
Some of the interesting special cases of the main results also established. 

\end{abstract}

\author{K.S. Nisar}
\address{K. S. Nisar : Department of Mathematics, College of Arts and
Science-Wadi Aldawaser, Prince Sattam bin Abdulaziz University, Alkharj, Riyadh
region 11991, Saudi Arabia}
\email{ksnisar1@gmail.com, n.sooppy@psau.edu.sa}

\maketitle

\section{Introduction}

The familiar generalized hypergeometric function $_{p}F_{q}$ is defined as
follows (see,\cite{Foxc, Rainville}): 
\begin{equation}  \label{eqn-1-hyper}
_{p}F_{q}\left[ 
\begin{array}{r}
\left( a _{p}\right) ; \\ 
\left( c _{q}\right) ;%
\end{array}
x\right] =\sum\limits_{n=0}^{\infty }\frac{\Pi _{j=1}^{p}\left( a_{j}\right)
_{n}}{\Pi _{j=1}^{q}\left( c _{j}\right) _{n}}\frac{x^{n}}{n! },
\end{equation}
\begin{equation*}
(p\leq q,\, x \in \mathbb{C};\,\, p=q+1,\, \left\vert x\right\vert <1),
\end{equation*}
where $(a)_{n}$ denoted by Pochhammer symbol given by
$$(a)_{n}=a(a+1)_{n-1}; (a)_{0}=1.$$

In particular, if $p=2$ and $q=1$,  \eqref{eqn-1-hyper} reduced to Gaussian hypergeometric function,
\begin{equation}\label{2F1}
{}_2F_{1}(a, b; c; x)=\sum_{k=0}^{\infty}\frac{(a)_{n}(b)_{n}}{(c)_{n}}\frac{x^{n}}{n!}.
\end{equation}

The polynomial $M_{n}^{\left( p,q\right) }\left( x\right) $ is the solution
of the differential equation%
\begin{equation}
x\left( x+1\right) y_{n}^{\prime \prime }\left( x\right) +\left( 2-p\right)
x+\left( 1+q\right) y_{n}^{\prime }\left( x\right) -n\left( n-1+p\right)
y_{n}\left( x\right) =0,  \label{Eq1}
\end{equation}%
is defined by%
\begin{equation}
M_{n}^{\left( p,q\right) }\left( x\right) =\left( -1\right)
^{n}n!\sum_{k=0}^{n}\left( 
\begin{array}{c}
p-n-1 \\ 
k%
\end{array}%
\right) \left( 
\begin{array}{c}
q+n \\ 
n-k%
\end{array}%
\right) \left( -x\right) ^{k}.  \label{Eq2}
\end{equation}%
These polynomials are orthogonal on $\left[ 0,\infty \right] $ with respect
to the weight function $w_{p,q}\left( x\right) =x^{q}\left( 1+x\right)
^{-\left( p+q\right) }$ if and only if $p>2n+1$ and $q>-1.$

The relation between $M_{n}^{\left( p,q\right) }\left( x\right) $ and
\eqref{2F1} can be expressed as:%
\begin{equation}
M_{n}^{\left( p,q\right) }\left( x\right) =\left( -1\right) ^{n}n!\left( 
\begin{array}{c}
q+n \\ 
n%
\end{array}%
\right) \text{ }_{2}F_{1}\left( -n,n+1-p;q+1;-x\right).   \label{Eq3}
\end{equation}%
The relation between Jacobi polynomial $P_{n}^{\left( \alpha ,\beta \right) }
$ and $M_{n}^{\left( p,q\right) }$ is given by 
\begin{equation}
M_{n}^{\left( p,q\right) }\left( x\right) =\left( -1\right)
^{n}n!P_{n}^{\left( a,-p-q\right) }\left( 2x+1\right),   \label{Eq4}
\end{equation}%
which can be expressed as 
\begin{equation}
P_{n}^{\left( a,-p-q\right) }\left( x\right) =\frac{\left( -1\right) ^{n}}{n!%
}M_{n}^{\left( -p-q,p\right) }\left( \frac{x-1}{2}\right).   \label{Eq5}
\end{equation}%
For more details about Jacobi polynomials and related results, one can refer \cite{Malik, Malik1, Malik2, Malik3}.



\vskip3mm 
Then the generalized fractional integral operators involving the Appell
functions $F_{3}$ are defined for $\delta $, $\delta ^{\prime }$, $\mu $, $\mu ^{\prime }$, $%
\epsilon \in \mathbb{C}$ with $\Re (\epsilon )>0$ and $x\in \mathbb{R}^{+}$as follows:
\begin{equation}
\left( I_{0{+}}^{\delta ,\delta ^{\prime },\mu ,\mu ^{\prime },\epsilon
}f\right) (x)=\frac{x^{-\delta }}{\Gamma (\epsilon )}\int_{0}^{x}(x-t)^{%
\epsilon -1}t^{-\delta ^{\prime }}F_{3}\left( \delta ,\delta ^{\prime },\mu
,\mu ^{\prime };\epsilon ;1-\frac{t}{x},1-\frac{x}{t}\right) f(t)\,\mathrm{d}t,
\label{Int-1}
\end{equation}%
and 
\begin{equation}
\left( I_{-}^{\delta ,\delta ^{\prime },\mu ,\mu ^{\prime },\epsilon
}f\right) (x)=\frac{x^{-\delta ^{\prime }}}{\Gamma (\epsilon )}%
\int_{x}^{\infty }(t-x)^{\epsilon -1}t^{-\delta }F_{3}\left( \delta ,\delta
^{\prime },\mu ,\mu ^{\prime };\epsilon ;1-\frac{t}{x},1-\frac{x}{t}\right)
f(t)\,\mathrm{d}t.  \label{Int-2}
\end{equation}%
The integral operators of the types \eqref{Int-1} and %
\eqref{Int-2} have been introduced by Marichev \cite{Marichev} and later
extended and studied by Sagio and Maeda \cite{Saigo}. Recently, many researchers (see, \cite{Baleanu, Kumar, SRN, Purohit}) have studied the image formulas for MSM fractional integral operators involving various special functions.

The corresponding fractional differential operators have their respective
forms: 
\begin{equation}
\left( D_{0+}^{\delta ,\delta ^{\prime },\mu ,\mu ^{\prime },\epsilon
}f\right) \left( x\right) =\left( \frac{\mathrm{d}}{\mathrm{d}x}\right) ^{%
\left[ \Re \left( \epsilon \right) \right] +1}\left( I_{0+}^{-\delta ^{\prime
},-\delta ,-\mu ^{\prime }+\left[ \Re \left( \epsilon \right) \right] +1,-\mu
,-\epsilon +\left[ \Re \left( \epsilon \right) \right] +1}f\right) \left(
x\right)  \label{Dif-1}
\end{equation}%
and 
\begin{equation}
\left( D_{-}^{\delta ,\delta ^{\prime },\mu ,\mu ^{\prime },\epsilon
}f\right) \left( x\right) =\left( -\frac{\mathrm{d}}{\mathrm{d}x}\right) ^{%
\left[ \Re \left( \epsilon \right) \right] +1}\left( I_{-}^{-\delta ^{\prime
},-\delta ,-\mu ^{\prime },-\mu +\left[ \Re \left( \epsilon \right) \right]
+1,-\epsilon +\left[ \Re \left( \epsilon \right) \right] +1}f\right) \left(
x\right) .  \label{Dif-2}
\end{equation}%

Here, we recall the following results (see \cite%
{Kilbas-itsf, Saigo}):

\begin{lemma}
\label{lem-1} Let $\delta ,\delta ^{\prime },\mu ,\mu ^{\prime },\epsilon
,\tau \in \mathbb{C}$ be such that $Re{(\epsilon)}>0$ and 
\begin{equation*}
\Re{(\tau)}>\max \{0,\Re{(\delta-\delta^{\prime }-\mu-\epsilon)},\Re{%
(\delta^{\prime }-\mu^{\prime })}\}.
\end{equation*}%
then there exists the relation 
\begin{equation}
\left( I_{0+}^{\delta ,\delta ^{\prime },\mu ,\mu ^{\prime },\epsilon
}\;t^{\tau -1}\right) (x)=\frac{\Gamma \left( \tau \right) \Gamma \left(
\tau +\epsilon -\delta -\delta ^{\prime }-\mu \right) \Gamma \left( \tau
+\mu {^{\prime }}-\delta {^{\prime }}\right) }{\Gamma \left( \tau +\mu {%
^{\prime }}\right) \Gamma \left( \tau +\epsilon -\delta -\delta {^{\prime }}%
\right) \Gamma \left( \tau +\epsilon -\delta {^{\prime }}-\mu \right) }%
x^{\tau -\delta -\delta ^{\prime }+\epsilon -1}
\end{equation}%
\end{lemma}

\begin{lemma}
\label{lem-2} Let $\delta $, $\delta ^{\prime }$, $\mu $, $\mu ^{\prime }$%
, $\epsilon $, $\tau \in \mathbb{C}$ such that $\Re (\epsilon )>0$ and 
\begin{equation*}
\Re {(\tau )}>\max \{\Re {(\mu )},\,\Re {(-\delta -\delta ^{\prime
}+\epsilon )},\,\Re {(-\delta -\mu ^{\prime }+\epsilon )}\},
\end{equation*}%
then  
\begin{equation}
\left( I_{-}^{\delta ,\delta ^{\prime },\mu ,\mu ^{\prime },\epsilon
}t^{-\tau }\right) (x)  \label{Lem2-Eq2} \\
=\frac{\Gamma \left( -\mu +\tau \right) \Gamma \left( \delta +\delta
^{\prime }-\epsilon +\tau \right) \Gamma \left( \delta +\mu {\ ^{\prime }}%
-\epsilon +\tau \right) }{\Gamma \left( \tau \right) \Gamma \left( \delta
-\mu +\tau \right) \Gamma \left( \delta +\delta {^{\prime }+\mu }%
^{^{\prime }}-\epsilon +\tau \right) }x^{-\delta -\delta ^{\prime }+\epsilon
-\tau }.  \notag
\end{equation}
\end{lemma}
Also, we need the following lemmas \cite{Kilbas-itsf}:
\begin{lemma}\label{lem-3}
Let $\delta, \mu, \epsilon \in \mathbb{C}$ be such that $\Re(\delta)>0, \Re(\tau)>max[0, \Re(\mu-\epsilon)]$
then
\begin{equation}\label{lem3-eqn1}
\left( I_{0+}^{\delta ,\mu,\epsilon
}\;t^{\tau -1}\right) (x)=\frac{\Gamma \left( \tau \right) \Gamma \left(
\tau +\epsilon -\mu \right)}{\Gamma \left( \tau -\mu\right) \Gamma \left( \tau +\epsilon +\delta
\right)}%
x^{\tau -\mu -1}.
\end{equation}%
\end{lemma}

In particular,
\begin{equation}\label{lem3-eqn2}
\left( I_{0+}^{\delta
}\;t^{\tau -1}\right) (x)=\frac{\Gamma \left( \tau \right)}{\Gamma \left( \tau+\delta
\right)}%
x^{\tau+ \delta-1}, \Re(\delta)>0, \Re(\tau)>0,
\end{equation}%

\begin{equation}\label{lem3-eqn3}
\left( I_{\epsilon,\delta}^{+
}\;t^{\tau -1}\right) (x)=\frac{\Gamma \left( \tau+\epsilon \right)}{\Gamma \left( \tau+\delta+\epsilon
\right)}%
x^{\tau+ \delta-1}, \Re(\delta)>0, \Re(\tau)>\Re(\epsilon).
\end{equation}%

\begin{lemma}\label{lem-4}
Let $\delta, \mu, \epsilon \in \mathbb{C}$ be such that $\Re(\delta)>0, \Re(\tau)<1+min[\Re(\mu), \Re(\epsilon)]$
then 
\begin{equation}\label{lem4-eqn1}
\left( I_{-}^{\delta ,\mu,\epsilon
}\;t^{\tau -1}\right) (x)=\frac{\Gamma \left( \mu-\tau+1 \right) \Gamma \left(
\epsilon -\tau+1 \right)}{\Gamma \left( 1-\tau\right) \Gamma \left( \delta+\mu+\epsilon-\tau+1
\right)}%
x^{\tau -\mu -1}.
\end{equation}%
\end{lemma}

In particular,
\begin{equation}\label{lem4-eqn2}
\left( I_{-}^{\delta
}\;t^{\tau -1}\right) (x)=\frac{\Gamma \left( 1-\delta-\tau \right)}{\Gamma \left(1- \tau
\right)}%
x^{\tau+ \delta-1}, 0<\Re(\delta)<1-\Re(\tau),
\end{equation}%

\begin{equation}\label{lem4-eqn3}
\left( K_{\epsilon,\delta}^{+
}\;t^{\tau -1}\right) (x)=\frac{\Gamma \left(\epsilon-\tau+1 \right)}{\Gamma \left(\delta+\epsilon-\tau+1
\right)}%
x^{\tau-1}, \Re(\tau)<1+\Re(\epsilon).
\end{equation}%

The main aim of this paper is to apply the generalized operators of fractional
calculus for the Jacobi type orthogonal polynomials in order to get certain new
image formulas. The basic definitions and results of fractional calculus,
one may refer to \cite{Kim, Kiryakova, Miller, Srivastava2}.

\section{Fractional integrals of Jacobi type orthogonal polynomials}
In this section, we derive the following theorems,
\begin{theorem} \label{Th1}
Let $\delta ,\delta ^{\prime },\mu ,\mu ^{\prime },\epsilon ,\tau
\in \mathbb{C}$ be such that $Re{(\epsilon )}>0$ and 
\begin{equation*}
\Re {(\tau )}>\max \{0,\Re {(\delta -\delta ^{\prime }-\mu -\epsilon )}%
,\Re {(\delta ^{\prime }-\mu ^{\prime })}\}.
\end{equation*}%
then 
\begin{eqnarray*}
&&\left( I_{0+}^{\delta ,\delta ^{\prime },\mu ,\mu ^{\prime },\epsilon
}t^{\tau -1}~M_{n}^{\left( p,q\right) }\left( t\right) \right) (x) \\
&=&x^{\tau -\delta -\delta ^{^{\prime }}+\epsilon -1}\frac{\left( -1\right)
^{n}\Gamma \left( q+n+1\right)\Gamma(\tau)\Gamma(\tau+\delta-\delta'-\mu)\Gamma(\tau+\mu'-\delta') }{ \Gamma \left(q+1\right)\Gamma(\tau+\mu')\Gamma(\tau+\epsilon-\delta-\delta')\Gamma(\tau+\epsilon-\delta'-\mu) } \\
&&\times _{5}F _{4}\left[ 
\begin{array}{c}
1+n-p, -n, \tau, \tau+\delta-\delta'-\mu, \tau+\mu'-\delta', \\ 
1+q, \tau+\mu', \tau+\epsilon-\delta-\delta', \tau+\epsilon-\delta'-\mu,%
\end{array}%
\Big|-x\right] .
\end{eqnarray*}
\end{theorem}

\begin{proof}
\label{Pf1} Applying the definition of $M_{n}^{\left( p,q\right) }\left(
x\right) $ given in \eqref{Eq2} and denote the left hand side by \emph{I}$_{1}$, we
get 
\begin{eqnarray*}
\emph{I}_{1} &=&\left( I_{0+}^{\delta ,\delta ^{\prime },\mu ,\mu ^{\prime
},\epsilon }t^{\tau -1}~M_{n}^{\left( p,q\right) }\left( t\right) \right) (x),
\\
&=&\left( I_{0+}^{\delta ,\delta ^{\prime },\mu ,\mu ^{\prime },\epsilon
}t^{\tau -1}\left( -1\right) ^{n}n!\sum\limits_{k=0}^{\infty }\left( 
\begin{array}{c}
p-n-1 \\ 
k%
\end{array}%
\right) 
\left( 
\begin{array}{c}
q+n \\ 
n-k%
\end{array}%
\right) \text{ }\left( -t\right) ^{k}\right) \left( x\right) ,
\end{eqnarray*}%
On interchanging the integration and summation, we obtain 
\begin{equation*}
\emph{I}_{1}=\left( -1\right) ^{n}n!\sum_{k=0}^{\infty }
\frac{\Gamma(p-n)}{k!\Gamma(p-n-k)}\frac{\Gamma(q+n+1)(-1)^{k}}{\Gamma(n-k+1)\Gamma(q+k+1)}\left( I_{0+}^{\delta ,\delta ^{\prime },\mu ,\mu ^{\prime },\epsilon
}(t^{\tau+k-1})\right)(x)
\end{equation*}%
Now, for any $k=0,1,2,...$.Since  
\begin{equation*}
\Re \left( \tau +k\right) \geq \Re \left( \tau \right) >\max \left[ 0,\Re
\left( \delta -\delta ^{^{\prime }}-\mu -\epsilon \right) ,\Re \left(
\delta ^{^{\prime }}-\mu ^{^{\prime }}\right) \right] ,
\end{equation*}%
and applying Lemma \ref{lem-1} and after some calculations, we get%
\begin{eqnarray*}
\emph{I}_{1} &=&\sum_{k=0}^{\infty }\frac{\left( -1\right) ^{n}\Gamma \left( q+n+1\right) (n-p)_{k}(-n)_{k}(-1)^{k}}{k!%
\Gamma \left(q+1+k\right)} \\
&&\times \frac{\Gamma \left( \tau +k\right) \Gamma \left( \tau +k+\delta
-\delta ^{^{\prime }}-\mu \right) \Gamma \left( \tau +k+\mu ^{\prime
}-\delta ^{^{\prime }}\right) }{\Gamma \left( \tau +k+\mu ^{^{\prime
}}\right) \Gamma \left( \tau +k+\epsilon -\delta -\delta ^{^{\prime
}}\right) \Gamma \left( \tau +k+\epsilon -\delta ^{\prime }-\mu \right) } \\
&&\times x^{\tau +k-\delta -\delta ^{^{\prime }}+\epsilon -1}.
\end{eqnarray*}%
In view of \eqref{eqn-1-hyper} and using the well known relation 
\begin{equation}\label{GR}
\Gamma(x+k)=(x)_{k}\Gamma(x),
\end{equation} 
we arrive desired result.
\end{proof}
If we set $\delta', \mu'=0, \delta=\delta+\mu, \mu=-\epsilon$ and $\epsilon=\delta$ in theorem \ref{Th1}, then we get the following result,
\begin{corollary}\label{Cor1} 
Let $\delta, \mu, \epsilon \in \mathbb{C}$ such that $\Re(\delta)>0, \Re(\tau)>max[0, \Re(\mu-\epsilon)]$, then
\begin{eqnarray*}
\left( I_{0+}^{\delta ,\mu ,\epsilon }t^{\tau -1}~M_{n}^{\left( p,q\right)
}\left( t\right) \right) (x)
=\frac{\left( -1\right) ^{n}\Gamma \left( \tau \right) \Gamma \left( \tau
+\epsilon -\mu \right) \Gamma \left( q+n+1\right) }{\Gamma \left( 1+q\right)
\Gamma \left( \tau -\mu \right) \Gamma \left( \tau +\delta +\epsilon \right) }%
x^{\tau -\mu -1} \\
\times _{4}F_{3}\left[ 
\begin{array}{c}
1+n-p,-n,\tau ,\tau +\epsilon -\mu,  \\ 
1+q,\tau-\mu ,\tau+\epsilon +\delta, 
\end{array}%
\left\vert 
\begin{array}{c}
-x%
\end{array}%
\right. \right] .
\end{eqnarray*}
which is theorem 2.1 of \cite{Malik}.
\end{corollary}

Substituting $\mu=-\delta$ in corollary \ref{Cor1} and using \eqref{Eq2}, we get ,
\begin{corollary}\label{Cor2} 
Let $\delta, \tau \in \mathbb{C}$ such that $\Re(\delta)>0, \Re(\tau)>0$, then
\begin{eqnarray*}
\left( I_{0+}^{\delta}t^{\tau-1}~M_{n}^{\left( p,q\right)
}\left( t\right) \right) (x)
=\frac{\left( -1\right) ^{n}\Gamma(1+q+n)\Gamma \left( \tau \right)}{\Gamma \left( 1+q\right)
\Gamma \left( \tau +\delta\right) }%
x^{\tau+\delta-1} \\
\times _{3}F_{2}\left[ 
\begin{array}{c}
1+n-p,-n,\tau, \\ 
1+q,\tau+\delta, 
\end{array}%
\left\vert 
\begin{array}{c}
-x%
\end{array}%
\right. \right] 
\end{eqnarray*}
\end{corollary}

Substituting $\mu=0$ in corollary \ref{Cor1}, we have,
\begin{corollary}\label{Cor3} 
Let $\delta, \epsilon, \tau \in \mathbb{C}$ such that $\Re(\delta)>0, \Re(\tau)>\Re(\epsilon)$, then
\begin{eqnarray*}
\left( I_{\epsilon, \delta}^{+}t^{\tau-1}~M_{n}^{\left( p,q\right)
}\left( t\right) \right) (x)
=\frac{\left( -1\right) ^{n}\Gamma(1+q+n)\Gamma \left( \tau+\epsilon \right)}{\Gamma \left( 1+q\right)
\Gamma \left( \tau +\delta+\epsilon\right) }%
x^{\tau-1} \\
\times _{3}F_{2}\left[ 
\begin{array}{c}
1+n-p,-n,\tau+\epsilon,  \\ 
1+q,\tau+\delta+\epsilon, 
\end{array}%
\left\vert 
\begin{array}{c}
-x%
\end{array}%
\right. \right] 
\end{eqnarray*}
\end{corollary}

\begin{theorem}
\label{Th2} Let $\delta ,\delta ^{\prime },\mu ,\mu ^{\prime },\epsilon
,\tau\in \mathbb{C}$,   and 
\begin{equation*}
\Re(\epsilon)>0, \Re {(\tau )}>\max \{\Re(\mu), \Re {(-\delta -\delta ^{\prime }+\epsilon^{\prime } )},\Re {%
(-\delta-\mu ^{\prime }+\epsilon )},
\end{equation*}%
then 
\begin{eqnarray*}
&&\left( I_{-}^{\delta ,\delta ^{\prime },\mu ,\mu ^{\prime },\epsilon
}t^{-\tau }~M_{n}^{\left( p,q\right) }\left( \frac{1}{t}\right) \right) (x) \\
&=&x^{-\delta -\delta ^{^{\prime }}+\epsilon -\tau}\frac{\left( -1\right)
^{n}\Gamma \left( q+n+1\right)\Gamma(\tau-\mu)\Gamma(\delta+\delta^{\prime}-\epsilon+\tau)\Gamma(\delta+\mu^{\prime}-\epsilon+\tau) }{\Gamma \left( \tau\right) \Gamma \left(q+1\right)\Gamma(\delta-\mu+\tau)\Gamma(\delta+\delta^{\prime}+\mu^{\prime}-\epsilon+\tau) } \\
&&\times _{5}\Psi _{4}\left[ 
\begin{array}{c}
1+n-p, -n, -\mu+\tau, \delta+\delta^{\prime}-\epsilon+\tau, \delta+\mu^{\prime}-\epsilon+\tau , \\ 
1+q, \tau, \delta-\mu+\tau, \delta+\delta^{\prime}+\mu^{\prime}-\epsilon+\tau ,%
\end{array}%
\Big|-\frac{1}{x}\right] .
\end{eqnarray*}
\end{theorem}

\begin{proof}\label{Pf2}  
 Applying the definition of $M_{n}^{\left( p,q\right) }\left(
x\right) $ in \ref{Eq2} and denote the left hand side by \emph{I}$_{2}$, we
get 
\begin{eqnarray*}
\emph{I}_{2} &=&\left( I_{-}^{\delta ,\delta ^{\prime },\mu ,\mu ^{\prime },\epsilon
}t^{-\tau }~M_{n}^{\left( p,q\right) }\left( \frac{1}{t}\right) \right) (x)
\\
&=&\left( I_{-}^{\delta ,\delta ^{\prime },\mu ,\mu ^{\prime },\epsilon
}t^{-\tau }\left( -1\right) ^{n}n!\sum\limits_{k=0}^{\infty }\left( 
\begin{array}{c}
p-n-1 \\ 
k%
\end{array}%
\right) 
\left( 
\begin{array}{c}
q+n \\ 
n-k%
\end{array}%
\right) \text{ }\left( -\frac{1}{t}\right) ^{k}\right) \left( x\right) ,
\end{eqnarray*}%
On interchanging the integration and summation, we obtain 
\begin{equation*}
\emph{I}_{2}=\left( -1\right) ^{n}n!\sum_{k=0}^{\infty }
\frac{\Gamma(p-n)}{k!\Gamma(p-n-k)}\frac{\Gamma(q+n+1)(-1)^{k}}{\Gamma(n-k+1)\Gamma(q+k+1)}\left( I_{-}^{\delta ,\delta ^{\prime },\mu ,\mu ^{\prime },\epsilon
}(t^{-\tau-k})\right)(x)
\end{equation*}%
Now, for any $k=0,1,2,...$.Since  
\begin{equation*}
\Re \left( \tau+k \right) >\max \left[ \Re(\mu),\Re
\left(-\delta -\delta ^{^{\prime }}-\epsilon'\right) ,\Re \left(
-\delta-\mu^{\prime }+\epsilon\right) \right] ,
\end{equation*}%
and applying Lemma \ref{lem-2} and some simple calculations, we get%
\begin{eqnarray*}
\emph{I}_{2} &=&\sum_{k=0}^{\infty }\frac{\left( -1\right) ^{n}\Gamma \left( q+n+1\right) (n-p)_{k}(-n)_{k}(-1)^{k}}{k!%
\Gamma \left(q+1+k\right)} \\
&&\times \frac{\Gamma \left( -\mu+\tau +k\right) \Gamma \left( \tau+\delta
+\delta ^{^{\prime }}-\epsilon+k \right) \Gamma \left( \tau+\mu ^{\prime
}+\delta-\epsilon+k\right) }{\Gamma \left( \tau +k\right) \Gamma \left( \tau +k+\delta-\mu\right) \Gamma \left( \delta+\delta^{\prime}+\mu^{\prime}-\epsilon+\tau+k \right) } \\
&&\times x^{-\tau-k-\delta -\delta ^{^{\prime }}+\epsilon}.
\end{eqnarray*}%
In view of \eqref{eqn-1-hyper} and \eqref{GR}, we reach the required result.
\end{proof}

If we set $\delta', \mu'=0, \delta=\delta+\mu, \mu=-\epsilon$ and $\epsilon=\delta$ in theorem \ref{Th2}, then we get the following result,
\begin{corollary}\label{Cor4} 
Let $\delta, \mu, \epsilon \in \mathbb{C}$ such that $\Re(\delta)>0, \Re(\tau)>1+min[\Re(\mu), \Re(\epsilon)]$, then
\begin{eqnarray}
\left( I_{-}^{\delta ,\mu ,\epsilon }t^{\tau -1}~M_{n}^{\left( p,q\right)
}\left( \frac{1}{t}\right) \right) (x)
=\frac{\Gamma(1+q+n)\Gamma(\mu-\tau+1)\Gamma(\epsilon-\tau+1)}{\Gamma(1+q)\Gamma(1-\tau)\Gamma(\delta+\mu+\epsilon-\tau+1)}x^{\tau-\mu-1} \notag\\
\times _{4}F_{3}\left[ 
\begin{array}{c}
1+n-p,-n,\mu-\tau+1, \epsilon-\tau+1,  \\ 
1+q,1-\tau ,\delta+\mu+\epsilon-\tau+1,
\end{array}%
\left\vert 
\begin{array}{c}
-\frac{1}{x}%
\end{array}%
\right. \right] 
\end{eqnarray}
\end{corollary}
which is theorem 2.2 of \cite{Malik}.

Substituting $\mu=-\delta$ in corollary \ref{Cor4} and using \eqref{Eq2}, we get,
\begin{corollary}\label{Cor5} 
Let $\delta, \epsilon, \tau \in \mathbb{C}$ such that $\Re(\delta)>0, \Re(\tau)<1+min[\Re(-\delta), \Re(\epsilon)]$ and $\tau+\delta \neq 1,2,...$, then
\begin{eqnarray*}
\left( I_{-}^{\delta}t^{\tau-1}~M_{n}^{\left( p,q\right)
}\left( \frac{1}{t}\right) \right) (x)
=\frac{\left( -1\right) ^{n}\Gamma(1+q+n)\Gamma \left(-\delta- \tau+1 \right)}{\Gamma \left( 1+q\right)
\Gamma \left(1-\tau\right) }%
x^{\tau+\delta-1} \\
\times _{3}F_{2}\left[ 
\begin{array}{c}
1+n-p,-n,-\delta-\tau+1,  \\ 
1+q,1-\tau, 
\end{array}%
\left\vert 
\begin{array}{c}
-\frac{1}{x}%
\end{array}%
\right. \right] 
\end{eqnarray*}
\end{corollary}

Substituting $\mu=0$ in corollary \ref{Cor4}, we have,
\begin{corollary}\label{Cor6} 
Let $\delta, \epsilon, \tau \in \mathbb{C}$ such that $\Re(\delta)>0, \Re(\tau)<1+min[0, \Re(\epsilon)]$ and let $\tau-\epsilon \neq 1,2,...$, then
\begin{eqnarray*}
\left( K_{\epsilon, \delta}^{-}t^{\tau-1}~M_{n}^{\left( p,q\right)
}\left( \frac{1}{t}\right) \right) (x)
=\frac{\left( -1\right) ^{n}\Gamma(1+q+n)\Gamma \left(\epsilon-\tau+1 \right)}{\Gamma \left( 1+q\right)
\Gamma \left(\delta+\epsilon-\tau+1\right) }%
x^{\tau-1} \\
\times _{3}F_{2}\left[ 
\begin{array}{c}
1+n-p,-n,\epsilon-\tau+1,  \\ 
1+q,\delta+\epsilon-\tau+1, 
\end{array}%
\left\vert 
\begin{array}{c}
-\frac{1}{x}%
\end{array}%
\right. \right] 
\end{eqnarray*}
\end{corollary}

\section{Fractional differentials of Jacobi type polynomials}

\label{Sec3}

In this section devoted to derive the MSM fractional differentiation of \eqref{Eq2}. We recall the following lemmas
(see \cite{Kataria}).

\vskip 3mm

\begin{lemma}
\label{lem-5} Let $\delta $, $\delta ^{\prime },$ $\mu ,$ $\mu ^{\prime }$%
, $\epsilon $, $\tau \in \mathbb{C}$ such that 
\begin{equation*}
\Re \left( \tau \right) >\max \left\{ 0,\Re \left( -\delta +\mu \right)
,\Re \left( -\delta -\delta ^{\prime }-\mu +\epsilon \right)
\right\} .
\end{equation*}%
Then 
\begin{eqnarray}
&&\left( D_{0+}^{\delta ,\delta ^{\prime },\mu ,\mu ^{\prime },\epsilon
}t^{\tau -1}\right) \left( x\right)  \label{D1} \\
&=&\frac{\Gamma \left( \tau \right) \Gamma \left( -\mu +\delta +\tau
\right) \Gamma \left( \delta +\delta ^{\prime }+\mu ^{^{\prime }}-\epsilon
+\tau \right) }{\Gamma \left( -\mu +\tau \right) \Gamma \left( \delta
+\delta ^{^{\prime }}-\epsilon +\tau \right) \Gamma \left( \delta +{\mu }%
^{^{\prime }}-\epsilon +\tau \right) }x^{\delta +\delta ^{\prime }-\epsilon
+\tau -1}.  \notag
\end{eqnarray}
\end{lemma}

\begin{lemma}
\label{lem-6} Let $\delta $, $\delta ^{\prime }$, $\mu $, $\mu ^{\prime }$%
, $\epsilon $, $\tau \in \mathbb{C}$ such that 
\begin{equation*}
\Re \left( \tau \right) >\max \left\{ \Re \left( -\mu ^{\prime }\right) ,\Re
\left( \delta ^{\prime }+\mu -\epsilon \right) ,\Re \left( \delta +\delta
^{^{\prime }}-\epsilon \right) +\left[ \Re \left( \epsilon \right) \right]
+1\right\} .
\end{equation*}%
Then the following formula holds true: 
\begin{eqnarray}
&&\left( D_{-}^{\delta ,\delta ^{\prime },\mu ,\mu ^{\prime },\epsilon
}t^{-\tau }\right) \left( x\right)  \label{D2} \\
&=&\frac{\Gamma \left( \mu ^{\prime }+\tau \right) \Gamma \left( -\delta
-\delta ^{\prime }+\epsilon +\tau \right) \Gamma \left( -\delta ^{\prime
}-\mu +\epsilon +\tau \right) }{\Gamma \left( \tau \right) \Gamma \left(
-\delta ^{\prime }+\mu ^{\prime }+\tau \right) \Gamma \left( -\delta
-\delta ^{\prime }-\mu +\epsilon +\tau \right) }x^{\delta +\delta ^{\prime
}-\epsilon -\tau }.  \notag
\end{eqnarray}
\end{lemma}

\begin{theorem}
\label{Th3} Let $\delta, \delta^{\prime}, \mu, \mu^{\prime}, \epsilon, \tau \in \mathbb{C}$ such that\\
 $\Re(\tau+k)>max\{0, \Re(-\delta+\mu), \Re(-\delta-\delta^{\prime}-\mu+\epsilon)\}$, 
then the following formula hold true: 
\begin{eqnarray*}
&&\left( D_{0+}^{\delta ,\delta ^{\prime },\mu ,\mu ^{\prime },\epsilon
}t^{\tau -1}~M_{n}^{(p, q)}(t)\right)
\left( x\right) \\
&=&\frac{(-1)^{n}\Gamma(q+n+1)\Gamma(\tau)\Gamma(\tau+\delta-\mu)\Gamma(\tau+\delta+\delta^{\prime}+\mu^{\prime}-\epsilon)}{\Gamma(q+1)\Gamma(\tau-\mu)\Gamma(\tau+\delta+\delta^{\prime}-\epsilon)\Gamma(\tau+\delta+\mu^{\prime}-\epsilon)}x^{\delta +\delta ^{^{\prime }}-\epsilon +\tau-1} \\
&&\times _{5}F _{4}\left[ 
\begin{array}{c}
1+n-p, -n, \tau, -\mu+\delta+\tau, \delta+\delta^{\prime}+\mu^{\prime}-\epsilon+\tau , \\ 
q+1, -\mu^{\prime}+\tau, \delta+\delta^{\prime}-\epsilon+\tau, \delta+\mu^{\prime}-\epsilon+\tau ,%
\end{array}%
\Big|-x\right],
\end{eqnarray*}
\end{theorem}

\begin{theorem}
\label{Th4} Let $\delta, \delta^{\prime}, \mu, \mu^{\prime}, \epsilon, \tau \in \mathbb{C}$ such that\\
 $\Re(\tau+k)>max\{0, \Re(-\mu'), \Re(\delta^{\prime}+\mu-\epsilon), \Re(\delta+\delta^{\prime}-\epsilon)+[\Re(\epsilon)]+1\}$, then the following formula hold true: 
\begin{eqnarray*}
&&\left( D_{-}^{\delta ,\delta ^{\prime },\mu ,\mu ^{\prime },\epsilon
}t^{-\tau}~M_{n}^{(p,q)}\left( \frac{1}{t}%
\right) \right) \left( x\right) \\
&=&\frac{(-1)^{n}\Gamma(q+n+1)\Gamma(\mu^{\prime }+\tau)\Gamma(-\delta-\delta^{\prime }+\epsilon+\tau)\Gamma(-\delta^{\prime }-\mu+\epsilon+\tau)}{\Gamma(q+1)\Gamma(\tau)\Gamma(-\delta^{\prime }+\mu^{\prime }+\tau)\Gamma(-\delta-\delta^{\prime }-\mu+\epsilon+\tau)} x^{\delta+\delta'-\epsilon-\tau}\\
&&\times _{5}F_{4}\left[ 
\begin{array}{c}
1+n-p, -n, \tau+\mu^{\prime }, \tau-\delta-\delta^{\prime }+\epsilon, \tau-\delta^{\prime }-\mu+\epsilon , \\ 
\tau, \tau-\delta^{\prime }+\mu^{\prime }, \tau+\epsilon-\delta-\delta^{\prime }-\mu, q+1,
\end{array}%
\Big|-\frac{1}{x}\right],
\end{eqnarray*}
provided both the sides exists.
\end{theorem}

\section{Conclusion}
The generalized fractional integration and differentiation of of Jacobi type orthogonal polynomials are derived in this paper. Many known results (see \cite{Malik}, \cite{Kilbas-itsf}) reduced as the particular case of the main theorems. This concludes that the results obtained here are general in nature and can easily obtain various known results.

\end{document}